\documentclass{imsart}

\RequirePackage{amsthm,amsmath,amsfonts,amssymb}
\RequirePackage[authoryear]{natbib}
\RequirePackage[colorlinks,citecolor=blue,urlcolor=blue]{hyperref}
\RequirePackage{graphicx}

\startlocaldefs

\newtheorem{theorem}{Theorem}[section]
\newtheorem{lemma}[theorem]{Lemma}

\theoremstyle{remark}

\newtheorem{example}{Example}

\newcommand{\argmax}{\raisebox{-1.8mm}{${{\displaystyle \rm arg\,\! max}
                     \atop {\scriptstyle x \in [0,1]}}$}}

\newcommand{\aargmax}{\raisebox{-1.8mm}{${{\displaystyle \rm arg\,\! max}
                     \atop {\scriptstyle 1 \le i \le n}}$}}
\endlocaldefs

\begin{document}

\begin{frontmatter}
\title{Exact and asymptotic goodness-of-fit tests based on the maximum and its location of the empirical process}
\runtitle{An exact goodness-of-fit test}

\begin{aug}
\author{\fnms{Dietmar} \snm{Ferger}\ead[label=e1]{dietmar.ferger@tu-dresden.de}}


\affiliation{Technische Universit\"{a}t Dresden, Fakult\"{a}t Mathematik}

\address{Fakult\"{a}t Mathematik\\
 Technische Universit\"{a}t Dresden\\
 D-01069 Dresden,
Germany\\
\printead{e1}\\
}

\end{aug}

\begin{abstract}
The supremum of the standardized empirical process is a promising statistic for
testing whether the distribution function $F$ of i.i.d. real random variables
is either equal to a given distribution function $F_0$ (hypothesis) or $F \ge F_0$ (one-sided alternative).
Since \cite{r5} it is well-known that an affine-linear transformation of the suprema converge
in distribution to the Gumbel law as the sample size tends to infinity. This enables the
construction of an asymptotic level-$\alpha$ test. However, the rate of convergence is extremely slow.
As a consequence the probability of the type I error is much larger than $\alpha$ even for sample sizes
beyond $10.000$. Now, the standardization consists of the weight-function $1/\sqrt{F_0(x)(1-F_0(x))}$.
Substituting the weight-function by a suitable random constant leads to a new test-statistic, for which we
can derive the exact distribution (and the limit distribution) under the hypothesis.
A comparison via a Monte-Carlo simulation shows that the new test is uniformly better than the Smirnov-test and
an appropriately  modified test due to \cite{r20}. Our methodology also works for the two-sided alternative $F \neq F_0$.
\end{abstract}


\end{frontmatter}

\section{Introduction}
Let $X_1,\ldots,X_n$ be $n \in \mathbb{N}$ independent and identically distributed real random variables defined on a probability space $(\Omega, \mathcal{A}, \mathbb{P})$ and with common distribution function $F$. Throughout the paper it is assumed that $F$ is continuous.
Given a (continuous) distribution function $F_0$ we want to test the hypothesis  $H_0:F=F_0$ versus the alternative $H_1:F \neq F_0, F \ge F_0$.
If
$$
 F_n(x)=\frac{1}{n} \sum_{i=1}^n 1_{\{X_i \le x\}}, \; x \in \mathbb{R},
$$
is the empirical distribution function, then the well-known Smirnov-statistic is given by
\begin{equation} \label{Mn}
 M_n := \sup_{x \in \mathbb{R}} \{F_n(x)-F_0(x)\}= \max_{1 \le i \le n} \{\frac{i}{n}-F_0(X_{i:n})\},
\end{equation}
where $X_{i:n}$ denotes the $i-$th order statistic and the equality holds almost surely (a.s.). By the Glivenko-Cantelli Theorem
$D_n := F_n-F_0$ converges a.s. to $D = F-F_0$ uniformly on $\mathbb{R}$. In order to make the Smirnov-test more
sensitive for deviations of $F$ from $F_0$ in the tails of $F_0$ one could use
$$
 W_n := \sup_{x:0<F_0(x)<1}\frac{F_n(x)-F_0(x)}{\sqrt{F_0(x)(1-F_0(x))}}.
$$
Here,
$$
Q_n(x):=\frac{F_n(x)-F_0(x)}{\sqrt{F_0(x)(1-F_0(x))}} \rightarrow Q(x)=\frac{F(x)-F_0(x)}{\sqrt{F_0(x)(1-F_0(x))}}
$$
a.s. for every $x \in \mathbb{R}$ with $0<F_0(x)<1$. The following example shows that this approach seems to be very promising.

\begin{example} \label{exp1} Let $\tau \in \mathbb{R}$ and $\delta > 1$. Define $F=F_{\tau,\delta}$ by
$$
 F(x):= \left\{ \begin{array}{l@{\quad,\quad}l}
                 \delta F_0(x) & x \le \tau\\ \beta (F_0(x)-F_0(\tau))+ \delta F_0(\tau) & x > \tau,
               \end{array} \right.
$$
If $\delta F_0(\tau) < 1$ and $\beta$ satisfies (*) $\delta F_0(\tau)+\beta(1-F_0(\tau))=1$, then $F$ is a distribution function. Further, one verifies that
$$
 D(x)= \left\{ \begin{array}{l@{\quad,\quad}l}
                 (\delta-\beta)(1-F_0(\tau)) F_0(x) & x \le \tau\\ (\delta-\beta)F_0(\tau)(1-F_0(x)) & x > \tau,
               \end{array} \right.
$$
and thus
$$
 Q(x)=\left\{ \begin{array}{l@{\quad,\quad}l}
                 (\delta-\beta)(1-F_0(\tau)) \sqrt{F_0(x)/(1-F_0(x))} & x \le \tau\\ (\delta-\beta)F_0(\tau)\sqrt{(1-F_0(x))/F_0(x)} & x > \tau.
               \end{array} \right.
$$
Since $\delta>1$ and therefore by (*) $\beta<1$ we see that $\delta-\beta$ is positive and hence $D=F-F_0>0$. This shows that $F$ lies in the alternative $H_1$.
Moreover it follows that $\tau$ is a maximizing point of $D$ and $Q$ as well. In particular
\begin{align*}
M:= \max_{x \in \mathbb{R}} D(x) &= (\delta-\beta) F_0(\tau)(1-F_0(\tau))\\
                                 &< (\delta-\beta) \sqrt{F_0(\tau)(1-F_0(\tau))}=\sup_{x:0<F_0(x)<1}Q(x)=:S.
\end{align*}

Thus the supremum $S$ is larger than the maximum $M$ by the factor $1/\sqrt{F_0(\tau)(1-F_0(\tau))}$, which
increases to infinity as $F_0(\tau) \rightarrow 0$ or $F_0(\tau) \rightarrow 1$. Recall that we are interested in detecting
deviations in the tails of $F_0$. For instance if $F_0(\tau)=0.01$ then the supremum $S$ is about ten times larger than the maximum $M$, namely
$S=10.0504*M$. Now both tests reject the hypothesis for large values of $M_n \approx M$ and $W_n \approx S$, whence we strongly expect
that the $W_n$-test is much more likely to indicate the alternative than the Smirnov-test.
\end{example}

So, as far as the behaviour on the alternative is concerned the $W_n$-test should be the better candidate. However, a serious problem occurs when we want to
determine the critical values or $p$-values.  Here we need the exact or at least asymptotic distribution of the underlying test-statistics $M_n$ and $W_n$.
In case of $M_n$ the exact and the asymptotic distribution are known since the publication of \cite{r1}:
\begin{equation} \label{Smir}
 \mathbb{P}(M_n \le x)=1-\sum_{i=0}^{[n(1-x)]} x {n \choose i}(x+\frac{i}{n})^{i-1}(1-x-\frac{i}{n})^{n-i}, \quad 0<x<1,
\end{equation}
where $[\cdot]$ denotes the floor function. From the exact formula \cite{r1} deduces the asymptotic distribution as $n \rightarrow \infty$:
$$
 \mathbb{P}(\sqrt{n} M_n\le x) \rightarrow 1-\exp\{-2x^2\}, \; x \ge 0.
$$
In contrast to $M_n$ an explicit expression for the distribution function of $W_n$ for finite sample size $n \in \mathbb{N}$ is not known in the literature.
Even worse, \cite{r2} shows that $W_n$ converges to infinity in probability, whence the construction of an asymptotic level-$\alpha$ test fails.
But there is a way out by the following limit theorem of \cite{r5}:
\begin{equation} \label{Jaeschke}
 \mathbb{P}(\sqrt{n} W_n\le \frac{x+b_n}{a_n}) \rightarrow e^{-e^{-x}} \; \forall \; x \in \mathbb{R},
\end{equation}
where $a_n=\sqrt{2 \log \log n}$ and $b_n=2 \log \log n+ \frac{1}{2}\log \log \log n -\frac{1}{2} \log \pi$. Thus
$$
c_{\alpha,n}=\frac{-\log(-\log(1-\alpha))+b_n}{a_n}
$$
yields that
$$
 \mathbb{P}(W_n> c_{\alpha,n}) \rightarrow \alpha
$$
and therefore the test with rejection region $\{W_n > c_{\alpha,n}\}$ is an asymptotic level-$\alpha$ test.
Unfortunately the convergence in (\ref{Jaeschke}) to an extreme-value distribution (Gumbel) is known to be very (very) slow. As a consequence there are poor approximations of the exact critical values by $c_{\alpha,n}$ even for large sample sizes $n$. In particular, the probability of the type I error is much larger than
the given level $\alpha$ of significance. For instance given $\alpha=0.1$ and $n=30, 300, 5000, 10.000$ the true probabilities of type I error are equal to $0.18913, 0.18524, 0.18388, 0.18211$, which are more than $1.8$ times greater than the required level of significance. If $\alpha=0.01$ the factor increases dramatically. Indeed even for the very large sample size
$n=10.000$ the true error probability is equal to $0.0671$ and thus more than six times bigger! From this point of view the $W_n$-test is unacceptable and we look for
an alternative. (We obtained the above probabilities by a Monte-Carlo simulation with $10^5$ replicates upon noticing that $W_n$ is distribution-free under the hypothesis $H_0$.)\\

Let us explain the basic idea for the construction of our new test. Assume the alternative $F$ is such that $D=F-F_0$ has a unique maximizing point $\tau$ (as for instance when $F$ is as in the above example).
To make the test still sensitive for small deviations of $F$ from $F_0$ in the tails of $F_0$ we
replace the weight function $1/\sqrt{F_0(x)(1-F_0(x))}$ by the constant $1/\sqrt{F_0(\tau)(1-F_0(\tau)}$. Since $\tau$ is unknown we estimate it by
$$
 \tau_n := \argmax \{F_n(x)-F_0(x\}) \stackrel{a.s.}{=} X_{R:n},
$$
where
$$
 R := \aargmax \frac{i}{n}-F_0(X_{i:n}) := \min\{1 \le k \le n: \frac{i}{n}-F_0(X_{i:n}) \le \frac{k}{n}-F_0(X_{k:n}) \; \forall \; 1 \le i \le n\}.
$$
By (\ref{Mn}) the estimator $\tau_n$ is the smallest maximizing point of $D_n=F_n-F_0$. According to Corollary 2.3 of \cite{r8} $\tau_n$ converges to $\tau$ a.s.
(In fact \cite{r8} considers minimizing points of $D_n$, but the arguments there can easily be carried over to maximizing points.)
Herewith it follows that
$$
 Q_n^*(x):= \frac{F_n(x)-F_0(x)}{F_0(\tau_n)(1-F_0(\tau_n))} \rightarrow Q^*(x)= \frac{F(x)-F_0(x)}{F_0(\tau)(1-F_0(\tau))}
$$
a.s. for each $x \in \mathbb{R}$. Our test rejects the hypothesis $H_0$ for large values of
$$
W_n^*:=\sup_{x\in \mathbb{R}} Q_n^*(x)=\frac{M_n}{\sqrt{F_0(X_{R:n})(1-F_0(X_{R:n}))}}=\frac{\frac{R}{n}-F_0(X_{R:n})}{\sqrt{F_0(X_{R:n})(1-F_0(X_{R:n}))}}.
$$
In the situation of the above example one has that
$$
 Q^*(x)= \left\{ \begin{array}{l@{\quad,\quad}l}
                 (\delta-\beta)\sqrt{(1-F_0(\tau))/F_0(\tau)} F_0(x)) & x \le \tau\\ (\delta-\beta)\sqrt{F_0(\tau)/(1-F_0(\tau))} (1-F_0(x)) & x > \tau
               \end{array} \right.
$$
and consequently $\sup_{x \in \mathbb{R}} Q^*(x) = \sup_{x \in \mathbb{R}} Q(x)=S$. Thus there is good hope that the $W_n^*$-test has a power on $H_1$ comparably as good as the
$W_n$-test. But in contrast to the latter we can determine not only the asymptotic but also the finite sample null-distribution of $W_n^*$.\\

Our methodology also works in case of the two-sided alternative $H_2: F \neq F_0$. Here, the Kolmogorov-Smirnov statistic
\begin{equation} \label{KS}
K_n=\sup_{x \in \mathbb{R}} |F_n(x)-F_0(x)|=\max_{1 \le i \le n} \max\{i/n-F_0(X_{i:n}),(i-1)/n-F_0(X_{i:n})\}
\end{equation}
again can be made more sensitive by weighting with $1/\sqrt{F_0(x)(1-F_0(x))}$. The resulting
statistic
$$
 V_n = \sup_{x:0<F_0(x)<1}\frac{|F_n(x)-F_0(x)|}{\sqrt{F_0(x)(1-F_0(x))}}
$$
exhibits the same problems as its one-sided counterpart $W_n$: There is no explicit formula for its finite sample size distribution and
$\sqrt{n}V_n \stackrel{\mathbb{P}}{\rightarrow} \infty$. However by \cite{r5}
\begin{equation} \label{Jaeschke2}
 \mathbb{P}(\sqrt{n} V_n\le \frac{x+b_n}{a_n}) \rightarrow e^{-2 e^{-x}} \; \forall \; x \in \mathbb{R},
\end{equation}
but again the rate of convergence is extremely slow. As an alternative test-statistic we introduce
$$
 V_n^*:= \frac{K_n}{\sqrt{F_0(\sigma_n)(1-F_0(\sigma_n))}},
$$
where $\sigma_n$ is the (smallest) maximizing point of $|F_n-F_0|$. Some elementary considerations show that $\sigma_n=X_{r:n}$, where
$$
 r = \aargmax \max\{\frac{i}{n}-F_0(X_{i:n}),F_0(X_{i:n})-\frac{i-1}{n}\}.
$$
In particular, $V_n^*$ can be computed by the formula
$$
 V_n^*=\frac{\max\{\frac{r}{n}-F_0(X_{r:n}),F_0(X_{r:n})-\frac{r-1}{n}\}}{\sqrt{F_0(X_{r:n})(1-F_0(X_{r:n}))}}.
$$
In contrast to its one-sided counterpart $W_n^*$ the exact distribution of $V_n^*$ is not known, but we are still able to derive its limit distribution.

By the quantile-transformation the statistics $M_n, W_n, K_n$ and $V_n$ are distribution-free under the hypothesis. We will see that our statistics $W_n^*$ and $V_n^*$ share this property, see Lemma \ref{a1}.\\

Notice that $V_n$ is a generalized Kolmogorov-Smirnov statistic
\begin{equation} \label{GK}
K_{n,\Phi}= \sup_{0<F_0(x)<1}|F_n(x)-F_0(x)|\Phi(F_0(x))
\end{equation}
and $W_n$ is a generalized Smirnov statistics
\begin{equation} \label{GS}
S_{n,\Phi}= \sup_{0<F_0(x)<1}\{F_n(x)-F_0(x)\}\Phi(F_0(x))
\end{equation}
pertaining to $\Phi(u)=1/\sqrt{u(1-u)}$. Utilization of generalized statistics as in (\ref{GK}) and (\ref{GS}) with $\Phi:(0,1) \rightarrow [0,\infty)$
is not new in the literature and in fact has a long history. We give a few examples.
\cite{r10} considers $\Phi(u)= 1_{[a,1]}(u)u^{-1}$ with fixed $0<a<1$.
His motivation was to measure the relative error (on the region $\{F_0 \ge a\})$ rather than the absolute error. He derives
the limit distributions of $\sqrt{n}K_{n,\Phi}$ and $\sqrt{n}S_{n,\Phi}$ under $F=F_0$, whereas the exact distribution of $K_{n,\Phi}$ and $S_{n,\phi}$ has been found by
\cite{r11} and \cite{r12}, respectively. Later on \cite{r17} extends his results to $\Phi(u)=1_{[a,b)}(u) u^{-1}, 0<a<b<1$. If one is interested in detecting differences in the tails $\{F_0\le a\})$ then
the $\Phi(u)=1_{(0,a]}u^{-1}$ can be used. The exact distribution of the pertaining
$$
 S_{n,\Phi}= \sup_{0<F(x)\le a}\frac{F_n(x)-F(x)}{F(x)}
$$
can be deduced from the exact result of \cite{r13} for $\sup_{0<F(x)\le a}\frac{F_n(x)}{F(x)}$ with fixed $0<a\le 1$.
Further examples are
$\Phi(u)$ is equal to
$$
1_{(a,b)}(u) \text{ or } 1-1_{(a,b)}(u) \text{ or } 1_{(a,b)}(u)(u(1-u))^{-1/2} \text{ with } 0<a<b<1.
$$
In all these examples one can detect discrepancies only over certain parts of the real line. Therefore it is more effective to put aside this restriction leading to $\Phi(u)=u^{-1}$, that is to
$$
 \sup_{0<F_0(x)<1}\frac{F_n(x)-F_0(x)}{F_0(x)}=L_n -1,
$$
where
\begin{equation} \label{Ln}
 L_n=\sup_{0<F_0(x)<1}\frac{F_n(x)}{F_0(x)}=\max_{1 \le i \le n} \frac{i}{n F_0(X_{i:n})}.
\end{equation}
It follows from \cite{r18}, confer also \cite{r19}, p.345, that under the hypothesis $H_0$ for each $n \in \mathbb{N}$,
\begin{equation} \label{Daniels}
 \mathbb{P}(L_n \le x)= 1-\frac{1}{x} \quad \text{for all } x \ge 1.
\end{equation}
All weight functions considered so far (up to $\Phi(u)=1/u$ and $\Phi(u)=1/\sqrt{u(1-u)}$) have in common that they are bounded. \cite{r14} use Donsker's theorem in combination with the Continuous Mapping Theorem
to show that the limit distributions are boundary-non-crossing probabilities of the Brownian bridge $B$, which with Doob's transformation can be rewritten as boundary-non-crossing probabilities of the Brownian motion. However, as the authors themselves state their formulas are such that ''the analytic difficulties of getting an explicit solution may be prohibitive''. For weight-functions which are not necessarily bounded we refer to \cite{r15}, Theorem 3.3 on p.220.
In case that $F=F_0$ is equal to the uniform distribution they show that
$$
 \sqrt{n} K_{n,\Phi}= \sqrt{n} \sup_{0<x<1}|F_n(x)-x|\Phi(x) \stackrel{\mathcal{D}}{\rightarrow} \sup_{0<x<1}|B(x)|\Phi(x).
$$
if and only if $1/\Phi$ belongs to a Chibisov-O'Reily class.

\cite{r20} propose the test-statistic
$$
 T_n=T_n(w):= \max\{w L_n,\sqrt{n} K_n,w U_n\}
$$ with
\begin{equation} \label{Un}
U_n:=\sup_{0<F_0(x)<1}\frac{1-F_n(x)}{1-F_0(x)}=\max_{1 \le i \le n}\frac{n-i}{n(1-F_0(X_{i:n}))}
\end{equation}
and $w \in \mathbb{R}$ is a positive weight. They prove that $(L_n,\sqrt{n} K_n,U_n)$ are asymptotically independent. More precisely, one has for all $a,c \ge 1$ and for all $b \ge 0$ that
\begin{equation} \label{MS}
 \mathbb{P}(L_n \le a, \sqrt{n} K_n \le b, U_n \le c) \rightarrow (1-1/a) G(b) (1-1/c), \; n \rightarrow \infty,
\end{equation}
where $G$ is the Kolmogorov-Smirnov distribution function. If for a given level  $\alpha \in (0,1)$ of significance
$x_\alpha :=G^{-1}((1-\alpha)^{1/3})$ and $w:=x_\alpha(1-(1-\alpha)^{1/3})$, then by (\ref{MS})
$$
 \mathbb{P}(T_n(w)>x_\alpha) \rightarrow \alpha.
$$
Thus the test with rejection region $\{T_n(w)>x_\alpha\}$ is an asymptotic level-$\alpha$ test. For instance $\alpha=0.05$ yields
$y_\alpha=1.544$ and $w=0.0261$. \cite{r20} give tables of the exact critical values $x_{\alpha,n}$ for selected sample sizes $n \in \mathbb{N}$ and
$\alpha \in \{0.1,0.05,0.01\}$. For the computation of these values they use that by (\ref{KS}), (\ref{Ln}) and (\ref{Un}) and freedomness of distribution under $H_0$
the probability $\mathbb{P}(T_n(w)\le x)$ can be rewritten as a rectangle probability for uniform order statistics. These in turn are calcultated by the recursion formula of \cite{r21}, confer also \cite{r19}, p. 362.

The counterpart of $T_n(w)$ designed for the one-sided alternative $H_1$ is
$$
 T^+_n:=T^+_n(w):= \max\{w L_n,\sqrt{n} M_n,w U_n\}.
$$
Carrying over the arguments in the proof of Theorem 1 in \cite{r20} one shows that
\begin{equation} \label{MS2}
 \mathbb{P}(L_n \le a, \sqrt{n} M_n \le b, U_n \le c) \rightarrow (1-1/a) (1-\exp(-2 b^2)) (1-1/c), \; n \rightarrow \infty.
\end{equation}

Following the procedure of \cite{r20} we put
\begin{equation} \label{w}
y_\alpha:=\sqrt{-\frac{1}{2} \log(1-(1-\alpha)^{1/3})} \quad \text{and} \quad w:= w_\alpha :=y_\alpha(1-(1-\alpha)^{1/3}).
\end{equation}
Then by (\ref{MS2}) it follows that
$$
  \mathbb{P}(T^+_n(w)>y_\alpha) \rightarrow \alpha
$$
and therefore $\{T^+_n(w)>y_\alpha\}$ is the rejection region of an asymptotic level-$\alpha$ test.
By (\ref{Mn}), (\ref{Ln}) and (\ref{Un}) one has  that under $H_0$
\begin{equation} \label{Tn+}
 \mathbb{P}(T^+_n(w) \le y)=\mathbb{P}(a_i \le X_{i:n} \le b_i \; \forall \; 1 \le i \le n),
\end{equation}
where $a_i=\max\{\frac{w}{y}\frac{i}{n},\frac{i}{n}-\frac{y}{\sqrt{n}}\}$ and $b_i=1-\frac{w}{y}(1-\frac{i}{n})$.
Moreover, the $X_{i:n}$ are the uniform order statistics. For the computations of the exact critical values $y_{\alpha,n}$ pertaining to
the sample size $n \in \mathbb{N}$ we prefer to use the formula of \cite{r22}:
\begin{equation} \label{Steck}
\mathbb{P}(a_i \le X_{i:n} \le b_i \; \forall \; 1 \le i \le n)= \text{det}(H_n),
\end{equation}
where the $ij-$th element $m_{i,j}$ of $H_n$ is equal to ${j \choose j-i+1}(b_i-a_j)_+^{j-i+1}$ or zero according as $j-i+1 \ge 0$ or not $(1 \le i,j \le n)$ and $(x)_+=\max\{x,0\}$. Thus $H_n$ is  an upper Hessenberg matrix, for which \cite{r23} prove the following recursion: det$(H_0)=1$, det$(H_1)=m_{1,1}$ and for $n \ge 2$:
\begin{equation} \label{rec}
 \text{det}(H_n)=m_{n,n} \; \text{det}(H_{n-1})+\sum_{r=1}^{n-1}\Big((-1)^{r-n} m_{r,n}\; \text{det}(M_{r-1}) \prod_{j=r}^{n-1} m_{j+1,j}\Big).
\end{equation}

With the help of (\ref{Tn+})-(\ref{rec}) we are able to calculate the exact critical values $y_{\alpha,n}$, which satisfy
$\mathbb{P}(T^+_n(w) > y_{\alpha,n}) = \alpha$ under $H_0$, see Table \ref{crit} below.\\

The paper is organized as follows: In the next section we derive the exact distribution of $W^*_n$ and the asymptotic distributions of
$\sqrt{n} W^*_n$  and $\sqrt{n} V^*_n$ under the hypothesis. In section 3 these results are used to determine the exact critical values
of the corresponding test statistics. In addition we present a table of exact critical values of the one-sided Mason-Schuenemeyer test (MS-test) based on $T^+_n(w)$. Afterwards we compare our new test with the Smirnov-test (S-test) and the MS-test test in a small simulation study.
It turns out that our test significantly performs better and surprisingly that the MS-test is inferior to the S-test. Finally, in the appendix we first prove that our test-statistics are distribution-free under the hypothesis. Moreover, it is shown that the argmax-functional appropriately defined on the Shorokhod-space is Borel-measurable and continuous on the subspace of all continuous functions with a unique maximizing point.
This result is essential for deriving the limit distributions via the Continuous Mapping Theorem.

\section{Exact and asymptotic null-distributions}

\begin{theorem} \label{exact} If $F=F_0$, then for all $x > 0$,
$$
 \mathbb{P}(W_n^* \le x)=1-\sum_{k=1}^n q_n[s(n^{-1}k,x),k]\;,
$$
where
$$
 s(c,x)=\frac{2c+x^2-x\sqrt{4c(1-c)+x^2}}{2(1+x^2)} \in (0,c), \; c \in (0,1],
$$
and
\begin{align*}
&q_n[z,k]={n-1 \choose k-1}(z\wedge\frac{k}{n})^k(1-z\wedge \frac{k}{n})^{n-k}\\
&- n^{-n} \sum_{i=0}^{k \wedge [nz]-1}\sum_{j=k}^n {n \choose j}{j \choose i}(n-nz \wedge k)^{n-j-1} (nz \wedge k -i-1)^{j-i}(j-nz \wedge k)(i+1)^{i-1}
\end{align*}
for all $z \in [0,1)$ and $k \in \{1,\ldots,n\}$.
The probability is equal to zero for all $x \le 0$.
\end{theorem}

\begin{proof} First notice that $W_n^*$ is distribution-free under $F=F_0$, see Lemma \ref{a1} in the appendix. Therefore we may assume that $F$ corresponds to the uniform distribution. It follows for $x>0$ that
$$
 \mathbb{P}(W_n^* \le x)=\mathbb{P}(\frac{\frac{R}{n}-X_{R:n}}{\sqrt{X_{R:n}(1-X_{R:n})}} \le x)=\sum_{k=1}^n \mathbb{P}(\frac{\frac{k}{n}-X_{R:n}}{\sqrt{X_{R:n}(1-X_{R:n})}} \le x, R=k ).
$$
Solving the inequality gives
$$
 \{\frac{\frac{k}{n}-X_{R:n}}{\sqrt{X_{R:n}(1-X_{R:n})}} \le x\}= \{X_{R:n} \ge s(n^{-1}k,x)\}
$$
and thus
$$
 \mathbb{P}(W_n^* \le x)= \sum_{k=1}^n \mathbb{P}(X_{R:n} \ge s(n^{-1}k,x),R=k)=\sum_{k=1}^n \mathbb{P}(X_{R:n} > s(n^{-1}k,x),R=k),
$$
where the last equality holds, because $X_{R:n}$ has a continuous distribution. In fact, it is uniformly distributed on$[0,1]$ by Theorem 3 of \cite{r6}. By complementation we arrive
at
$$
\mathbb{P}(W_n^* \le x)= 1- \sum_{k=1}^n \mathbb{P}(X_{R:n} \le s(n^{-1}k,x),R=k),
$$
which yields the desired result upon noticing that
$$
 \mathbb{P}(X_{R:n} \le z,R=k)= q_n[z,k]
$$
by \cite{r7}, p.53. If $x=0$, then $\mathbb{P}(W_n^* \le 0)=\sum_{k=1}^n \mathbb{P}(X_{R:n} = \frac{k}{n}, R=k)=0$. If $x<0$, then
the probability is also equal to zero, because $W_n^* \ge 0$ almost surely.
\end{proof}

Next we show that $\sqrt{n} W_n^*$ converges to the Maxwell-Boltzmann distribution as the sample size $n$ tends to infinity.

\begin{theorem} If $F=F_0$, then
$$
 H_n(x):=\mathbb{P}(\sqrt{n} W_n^* \le x) \rightarrow H(x)=2 \Phi(x)-\sqrt{\frac{2}{\pi}} x e^{-\frac{x^2}{2}} -1, \quad x \ge 0,
$$
where $\Phi$ is the distribution function of the standard normal law $N(0,1)$.
\end{theorem}

\begin{proof} Recall that $W_n^*$ is distribution-free by Lemma \ref{a1}. The basic idea is to write $\sqrt{n} W_n^*$ as a functional of the uniform empirical process $\alpha_n(t):=\sqrt{n}(F_n(t)-t), t \in [0,1]$.
To this end let $(D[0,1],s)$ be the Shorokhod-space and for every $f \in D[0,1]$ define $M(f):=\sup_{t \in [0,1]} f(t), A(f):=\{t \in [0,1]: \max\{f(t),f(t-)\}=M(f)\}$ with the convention $f(0-):=f(0)$. By Lemma \ref{a2} $A(f)$ is a non-empty compact subset of [0,1], whence the \emph{argmax-functional} $a(f):=\min A(f)$ is well defined. Actually we should call $a$ the \emph{argsup-functional}, since in general
it gives the smallest supremizing point of $f$. One verifies easily the simple property $a(c f)=a(f)$ for every positive constant $c$. Therefore
$\tau_n=a(\alpha_n)$. Similarly, $M(c f)=c M(f)$ and thus $\sqrt{n} M_n= M(\alpha_n)$. The functional $L:=(M,a):(D[0,1],s) \rightarrow \mathbb{R}^2$ is Borel-measurable by Lemma \ref{a3} and
continuous on the subset $C_u:=\{f \in C[0,1]: f \text{ has a unique maximizing point}\} \subseteq D[0,1]$. To see this
note that $M$ is continuous (even) on $C[0,1]$ and $a$ is continuous on $C_u$ by Lemma \ref{a4}. Since by Donsker's theorem
$\alpha_n \stackrel{\mathcal{D}}{\rightarrow} B$ in $(D[0,1],s)$, where $B$ is a Brownian bridge and $B \in C_u$ almost surely, an application of the
Continuous Mapping Theorem (CMT) yields that $(\sqrt{n} M_n,\tau_n) = L(\alpha_n) \stackrel{\mathcal{D}}{\rightarrow} L(B)=(M(B),a(B))$.
Let $h:\mathbb{R}^2 \rightarrow \mathbb{R}$ be defined by $h(x,y):=x/\sqrt{y(1-y)}$ for $(x,y) \in \mathbb{R} \times (0,1)=: G$ and $h(x,y):=17$ otherwise.
Obviously $h$ is continuous on $G$. Moreover, $a(B) \in (0,1)$ almost surely. (Indeed, $a(B)$ is uniformly distributed on $[0,1]$.) Thus
another application of the CMT gives
$$
 \sqrt{n} W_n^*=\frac{\sqrt{n} M_n}{\sqrt{\tau_n(1-\tau_n)}}=h(\sqrt{n} M_n,\tau_n) \stackrel{\mathcal{D}}{\rightarrow} h(M(B),a(B))=:Z^+.
$$
Now the assertion follows from Theorem 1.1 of \cite{r9}, which says that the distribution of $Z^+$ is equal to the Maxwell-Boltzmann distribution.
\end{proof}

Since $H_n(x)=\mathbb{P}(W_n^* \le x/\sqrt{n})$ we can compute the exact distribution function $H_n$ with Theorem \ref{exact}.
Figure \ref{figure1} shows that already for a small sample size there is a fairly good approximation.
\begin{figure}
\includegraphics[scale=0.4]{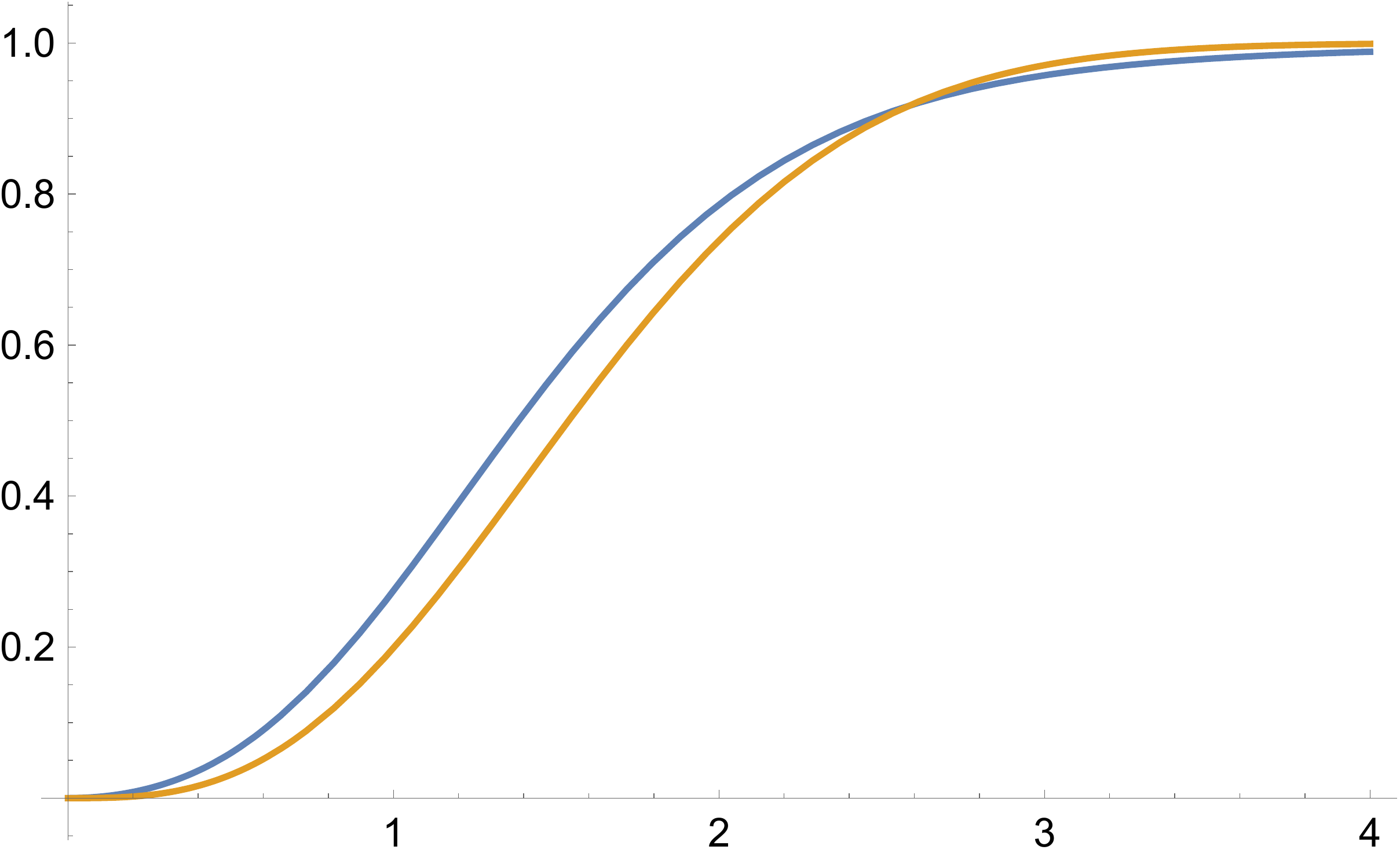}
\caption{Graphs of $H_{15}$ (blue) and $H$.}
\label{figure1}
\end{figure}

\begin{theorem} If $F=F_0$, then
$$
 \mathbb{P}(\sqrt{n}V_n^* \le x) \rightarrow G(x),
$$
where
\begin{eqnarray} \label{FR}
\lefteqn{G(x)= }\\
& & 16 \sum_{0\le j<l<\infty}(-1)^{j+l}\frac{\alpha_j\alpha_l}{\alpha_l^2-\alpha_j^2}[\frac{\Phi(\alpha_j x)-1/2}{\alpha_j}-\frac{\Phi(\alpha_l x)-1/2}{\alpha_l}]+ \hspace{4.1cm} \nonumber\\
& &  4 \sum_{0\le j < \infty} [\frac{\Phi(\alpha_j x)-1/2}{\alpha_j}-x \varphi(\alpha_j x)], \quad x \ge 0. \nonumber
\end{eqnarray}
Here, $\alpha_j = 2 j +1, j \in \mathbb{N}_0,$ and $\varphi$ denote the density of $N(0,1)$.
\end{theorem}

\begin{proof} Observe that $\sqrt{n} W_n^* = h \circ L(|\alpha_n|)$ and $|\alpha_n| \stackrel{\mathcal{D}}{\rightarrow} |B|$ in
$(D[0,1],s)$. Thus the CMT guarantees that
$$
 \sqrt{n} W_n^* \stackrel{\mathcal{D}}{\rightarrow} h(M(|B|),a(|B|))=:Z
$$
and the assertion follows from Theorem 1.1 of \cite{r9}.
\end{proof}

\section{Power investigations}
Recall that $H_n(x):=\mathbb{P}(\sqrt{n} W^*_n \le x)$ is given by Theorem \ref{exact}. Similarly, for $I_n(x):=\mathbb{P}(\sqrt{n} M_n(x) \le x)$
there is an explicit expression according to (\ref{Smir}). Finally, $J_n(x):=\mathbb{P}(T^+_n(w) \le x)$ can be computed via (\ref{Tn+})-(\ref{rec}).
For a given level $\alpha$ of significance let $z_{\alpha,n}, u_{\alpha,n}$ and $y_{\alpha,n}$ be the exact critical values of our test, the Smirnov-test and the MS-test, respectively. Thus these values are determined through
$H_n(z_{\alpha,n})=I_n(u_{\alpha,n})=J_n(y_{\alpha,n})=1-\alpha$. For $\alpha=0.05$ we provide a table of the critical values for some selected
sample sizes $n$. As to $T^+_n(w)$ recall that by (\ref{w}) the weight $w=w_\alpha=0.024205$.
\begin{table*}
\caption{Exact critical values for $\alpha=0.05$}
\label{crit}
\begin{tabular}{c|c|c|c}
\vspace{0.2cm}
$n$ & $z_{\alpha,n}$ & $u_{\alpha,n}$ & $y_{\alpha,n}$\\ \hline
 30 & 2.83457 & 1.19214 & 1.27950\\
 50 & 2.81185 & 1.20014 & 1.28827\\
100 & 2.79586 & 1.20856 & 1.29575\\
500 & 2.78631 & 1.21612 & 1.30498\\
1.000 & 2.78484 & 1.21869 & 1.30680\\
10.000& 2.79339 & 1.22238 & 1.31094\\
$\infty$ & 2.79548 & 1.22387 & 1.42782\\
\hline
\end{tabular}
\end{table*}
Table \ref{crit} shows that the asymptotic critical values of our new test (N-test) and the S-Test are fairly good even for small sample sizes, whereas the asymptotic value of the MS-Test is significantly larger than the exact values even for very large $n$. This indicates that the speed of convergence in (\ref{MS2}) seems to be rather slow.

Next we present the results of a small simulation study. Here we choose $F_0(x)=x$ on $[0,1]$ and $F=F_{\tau,\delta}$ as in
Example \ref{exp1}, that means the alternative $F$ on $[0,1]$ is the simple polygonal line through the points
$(0,0), (\tau,\delta \tau))$ and $(1,1)$. We fix $\tau=0.05$ and shortly write $F_\delta:=F_{0.05,\delta}$.
For $\delta \in [1,1/\tau)=[1,20)$ let
\begin{align*}
\beta^{(N)}_n(\delta)&:=\mathbb{P}_\delta(\sqrt{n}W^*_n > z_{\alpha,n}),\\
\beta^{(S)}_n(\delta)&:=\mathbb{P}_\delta(\sqrt{n}M_n > u_{\alpha,n}),\\
\beta^{(MS)}_n(\delta)&:=\mathbb{P}_\delta(T^+_n(w) > y_{\alpha,n})
\end{align*}
be the power-functions of the N-, S- and MS-test. (Under $\mathbb{P}_\delta$ the data $X_1,\ldots,X_n$ are i.i.d. with distribution function $F_\delta$). A Monte-Carlo simulation with $10^4$ replicates per grid-point yields the following results as displayed in Figure \ref{p30}-\ref{p500}. Here, the power functions $\beta^{(N)}_n, \beta^{(S)}_n$ or
$\beta^{(MS)}_n$ are represented by the blue, orange or green line, respectively. We see that for small (n=30), middle (n=100) and large (n=500) sample sizes the $N$-test with a clear distance is uniformly better than the $S$-test, which in turn is uniformly better than the $MS$-test.
The latter may come as a surprise, but it may be because of that the weighting by $1/F_0(x)$ and $1/(1-F_0(x))$ is unduly.

\begin{figure}
\includegraphics[scale=0.6]{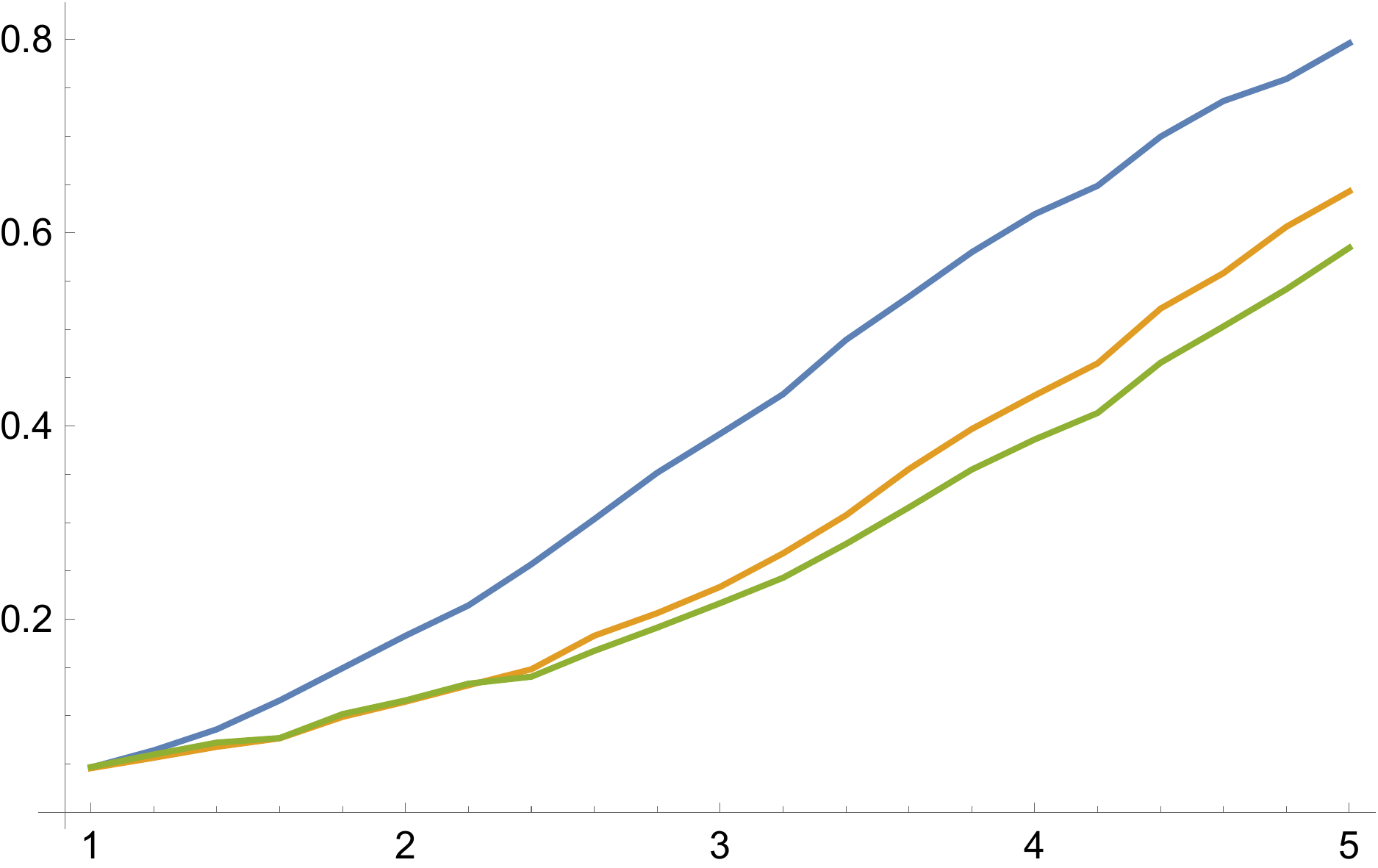}
\caption{Power for $n=30$}
\label{p30}
\end{figure}

\begin{figure}
\includegraphics[scale=0.6]{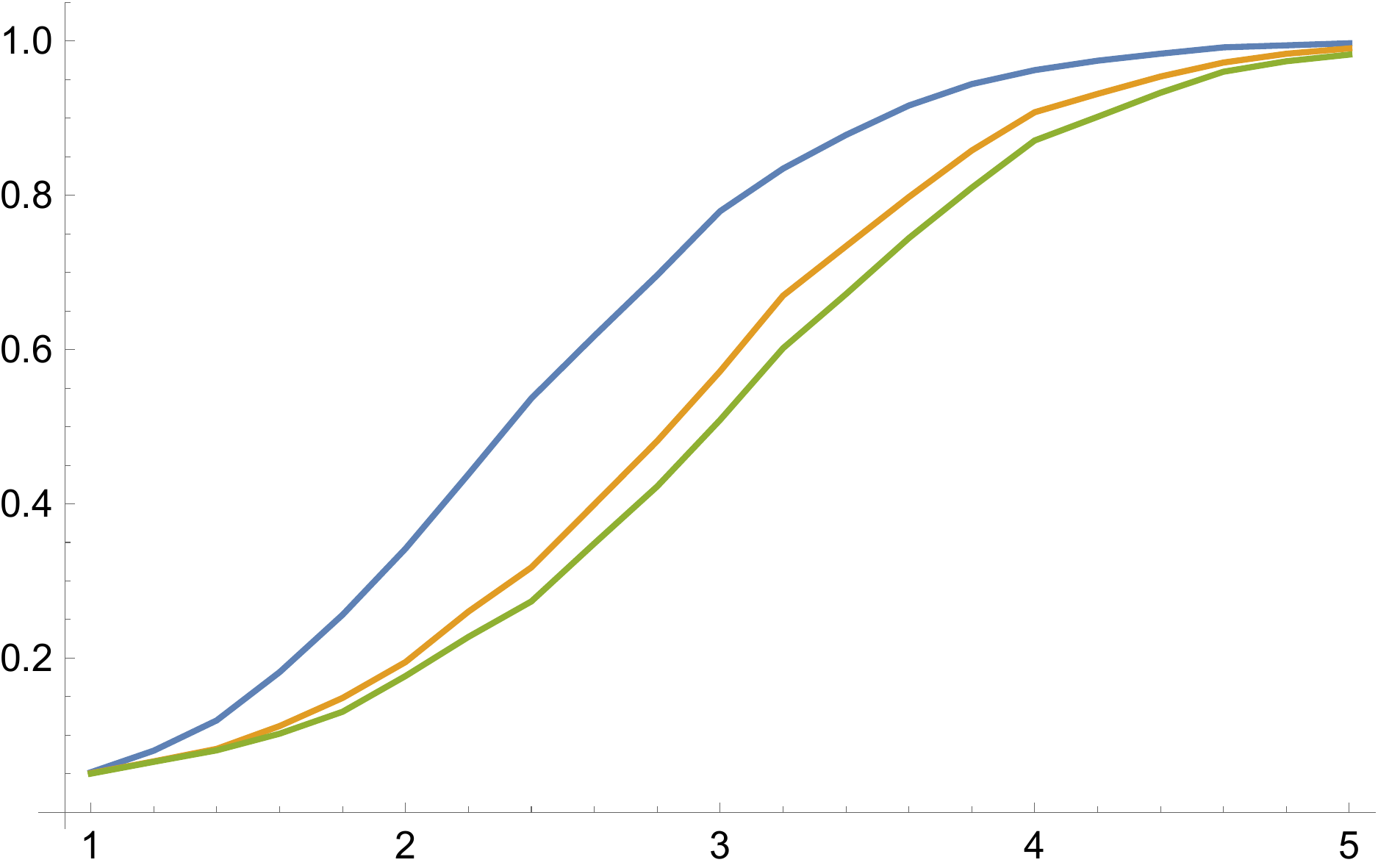}
\caption{Power for $n=100$}
\label{p100}
\end{figure}

\begin{figure}
\includegraphics[scale=0.6]{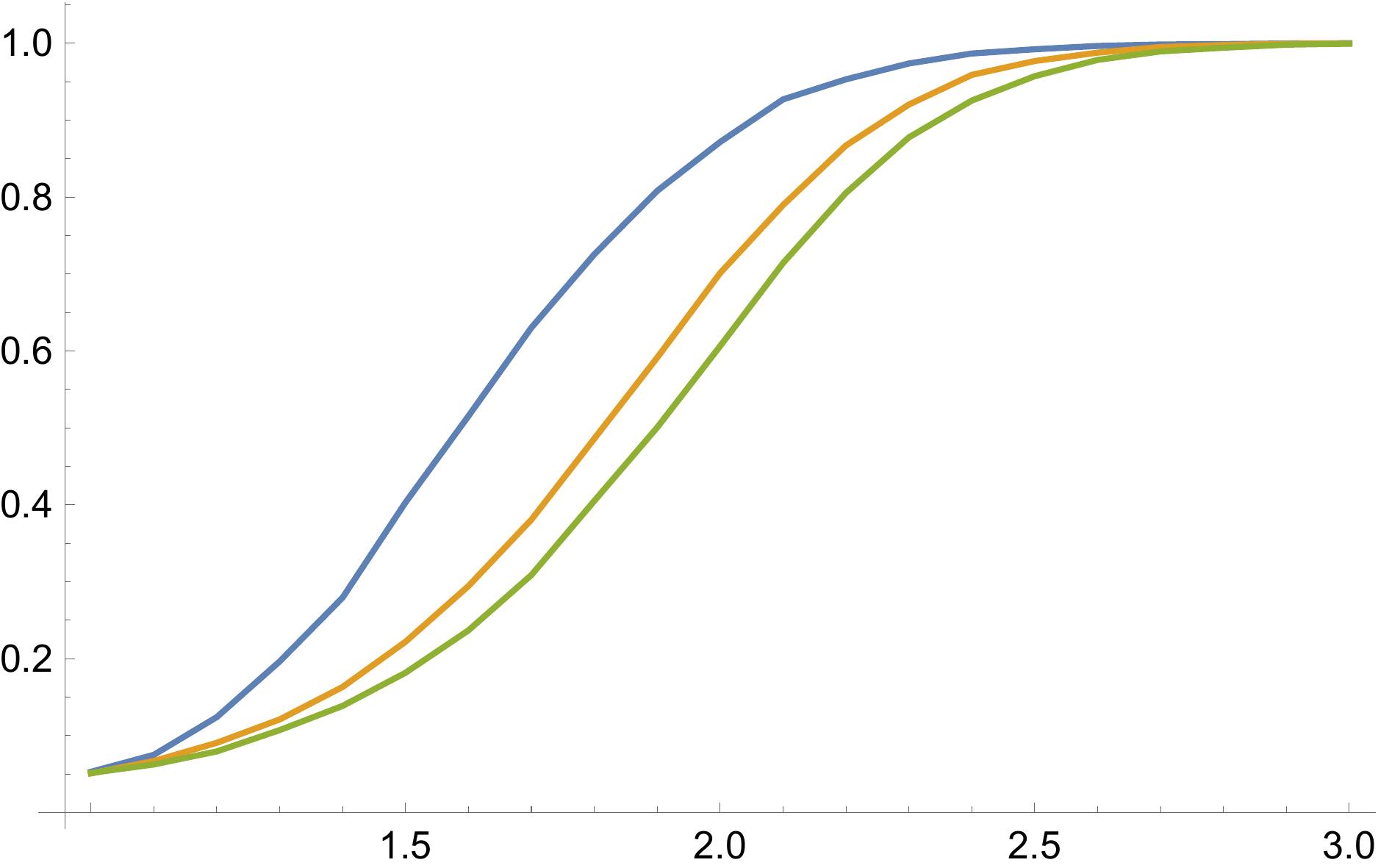}
\caption{Power for $n=500$}
\label{p500}
\end{figure}

In practical applications the statistician computes the $p$-value. For a given realization of the test-statistics $\sqrt{n} W^*_n, \sqrt{n} M_n$ and $T^+_n(w)$ the corresponding $p$-values are the (random) quantities $1-H_n(\sqrt{n} W^*_n), 1-I_n(\sqrt{n} M_n)$ and $1-J_n(T^+_n(w))$.

\begin{appendix}
\section*{Appendix}\label{appn}
If $F=F_0$, then our statistics have the shape

\begin{align} 
 W_n^* &=\frac{\frac{R}{n}-F(X_{R:n})}{\sqrt{F(X_{R:n})(1-F(X_{R:n}))}} \quad \text{ with } \label{wn}\\
       R&=\aargmax \frac{i}{n}-F(X_{i:n}) \label{R}
\end{align}
and
\begin{align} 
  V_n^*&=\frac{\max\{\frac{r}{n}-F(X_{r:n}),F(X_{r:n})-\frac{r-1}{n}\}}{\sqrt{F(X_{r:n})(1-F(X_{r:n}))}} \quad \text{ with } \label{Vn}\\
     r &= \aargmax \max\{\frac{i}{n}-F(X_{i:n}),F(X_{i:n})-\frac{i-1}{n}\}. \label{r}
\end{align}

\begin{lemma} \label{a1} If $F$ is continuous, then the distribution of $W_n^*$ in (\ref{wn}) does not depend on $F$.
The same holds for $V_n^*$ in (\ref{Vn}).
\end{lemma}

\begin{proof} By the quantile-transformation we can w.l.o.g. assume that $X_i=F^{-1}(U_i), 1 \le i \le n$.
By continuity of $F$ it follows that $X_{i:n}=F^{-1}(U_{i:n})$ for all indices $i$ and in particular
$X_{R:n}=F^{-1}(U_{R:n})$. As in the proof of Lemma A.2 in \cite{r8} one shows that $R=\aargmax \frac{i}{n}-U_{i:n}$.
Since by continuity $F \circ F^{-1}$ is the identity map we obtain from (\ref{wn}) and (\ref{R}) that
$$
 W_n^*=\frac{\frac{R}{n}-U_{R:n}}{U_{R:n}(1-U_{R:n})} \quad \text{ with } \quad R=\aargmax \frac{i}{n}-U_{i:n}.
$$
In the same manner one shows that
\begin{align} 
  V_n^*&=\frac{\max\{\frac{r}{n}-U_{r:n},U_{r:n}-\frac{r-1}{n}\}}{\sqrt{U_{r:n})(1-U_{r:n})}} \quad \text{ with } \label{vn}\\
     r &= \aargmax \max\{\frac{i}{n}-U_{i:n},U_{i:n})-\frac{i-1}{n}\} \label{r}
\end{align}
and the proof is complete.
\end{proof}

\begin{lemma} \label{a2} For every $f \in D[0,1]$ let $M(f):=\sup_{t \in [0,1]} f(t)$ and $A(f):=\{t \in [0,1]: \max\{f(t),f(t-)\}=M(f)\}$ with the convention $f(0-):=f(0)$. Then $A(f)$ is non-empty and compact. In particular, $a(f):=\min A(f)$ is well-defined. The statements remain true if
$[0,1]$ is replaced by any compact subinterval.
\end{lemma}

\begin{proof} Introduce $\bar{f}(x):=\max\{f(x),f(x-)\}$. Then by Lemmas 2.1 and 2.2 of \cite{r24} the function $\bar{f}$ is the upper semicontinuous regularization of $f$ and
$A(f)$ is equal to the set of all maximizing points of $\bar{f}$. Since $[0,1]$ is compact the latter set set is known
to be non-empty and compact.
\end{proof}

\begin{lemma} \label{a3} The functional $L:=(M,a):(D[0,1],s) \rightarrow \mathbb{R}^2$ is Borel-measurable.
\end{lemma}

\begin{proof} Since the Borel-$\sigma$ algebra $\mathcal{B}(\mathbb{R}^2)$ is equal to $\mathcal{B}(\mathbb{R})\otimes \mathcal{B}(\mathbb{R})$ it
suffices to show that $M$ and $a$ are Borel-measurable. By right-continuity of $f \in D[0,1]$ we have that $M(f)=\sup_{t \in [0,1] \cap \mathbb{Q}}f(t)$ and
therefore $M:(D[0,1],s)\rightarrow \mathbb{R}$ is Borel-measurable upon noticing that the Borel-$\sigma$ algebra on $(D[0,1],s)$ is generated by the
projections (evaluation maps), see Theorem 12.5 in \cite{r3}.

As to measurability of $a$ let $x \in [0,1]$. Then (with $D:=D[0,1])$ we have that
\begin{align}
\{f \in D: a(f)\le x\}& =\{f \in D:\sup_{t \in [0,x]} f(t) \ge \sup_{t \in [x,1]} f(t)\}  \label{seteq}\\
                       & =\{f \in D:\sup_{t \in [0,x]\cap \mathbb{Q}} f(t) \ge \sup_{t \in [x,1] \cap \mathbb{Q}} f(t)\}.\label{seteq2}
\end{align}
To see the equality (\ref{seteq}) assume that $a(f)\le x$, but $\sup_{t \in [0,x]} f(t) < \sup_{t \in [x,1]} f(t)$. Put $\tau:=a(f)$.
Then $f(\tau) \le \sup_{t \in [0,x]} f(t) < \sup_{t \in [x,1]} f(t) \le M(f)$. If $\tau >0$ we find a sequence $(s_k)$ with
$0 <s_k \uparrow \tau$, whence $f(\tau-) = \lim_{k \rightarrow \infty}f(s_k)$. In particular $s_k \le \tau \le x$ and so
$f(s_k) \le \sup_{t \in [0,x]} f(t)$ for each $k$. Taking the limit $k \rightarrow \infty$ yields $f(\tau-) \le \sup_{t \in [0,x]} f(t) < \sup_{t \in [x,1]} f(t) \le M(f)$. As a consequence $\max\{f(\tau),f(\tau-)\} < M(f)$, a contradiction to $\tau$ is a supremizing point. If $\tau=0$, then by convention $f(0-)=f(0)$ and $f(0)<M(f)$ as seen above. So we arrive at a contradiction also in this case.

For the other direction observe that $M(f)=\max\{\sup_{t \in [0,x]} f(t),\sup_{t \in [x,1]} f(t)\}=\sup_{t \in [0,x]} f(t)= \max\{f(\sigma), f(\sigma-)\}$ for some $\sigma \in [0,x]$ by Lemma \ref{a2}. Thus $\sigma \in A(f)$ and since $a(f)$ is the smallest supremizing point it follows that $a(f) \le \sigma \le x$ as desired. This shows equality (\ref{seteq}). The second equality (\ref{seteq2}) holds, because the respective suprema
coincide by right-continuity of $f$. Measurability now follows again by noticing that the Borel-$\sigma$ algebra on $(D[0,1],s)$ is generated by the
projections.
\end{proof}

\begin{lemma} \label{a4} The argmax-functional $a:(D[0,1],s) \rightarrow [0,1]$ is continuous on the class $C_u$ of all continuous functions
$f:[0,1] \rightarrow \mathbb{R}$ with unique maximizing point.
\end{lemma}

\begin{proof} Let $f \in C_u$ with unique minimizer $\tau=a(f)$ and let $(f_n)$ be a sequence in $D[0,1]$ such that $s(f_n,f) \rightarrow 0$.
Since $f$ is continuous we have that in fact
\begin{equation} \label{unifconv}
||f_n-f||=\sup_{0 \le t \le 1} |f_n(t)-f(t)| \rightarrow 0.
\end{equation}
Assume that $\tau \in (0,1)$. For an arbitrary $0< \epsilon \le \tau \wedge (1-\tau)$ introduce $U_\epsilon :=(\tau-\epsilon,\tau+\epsilon] \subseteq [0,1]$ with non-empty complement $V_\epsilon = [0,1] \setminus U_\epsilon = [0,\tau-\epsilon] \cup (\tau+\epsilon,1]$. Consider
$$
 m_\epsilon := \sup\{f(t): t \in V_\epsilon\}=\sup\{f(t): t \in \bar{V_\epsilon} \},
$$
where $\bar{V_\epsilon}$ is the closure of $V_\epsilon$ and equal to $[0,\tau-\epsilon] \cup [\tau+\epsilon,1]$. Here the second equality holds by continuity of $f$. Put $\delta_\epsilon :=\frac{1}{3}(f(\tau)-m_\epsilon)$. Then
$\delta_\epsilon > 0$, because otherwise $f(\tau)=m_\epsilon = f(\sigma)$ for some $\sigma \in \bar{V_\epsilon}$ upon noticing that $\bar{V_\epsilon}$ is compact. Consequently, $\sigma$  is a maximizing point, which differs from $\tau$, because it does not lie in $(\tau-\epsilon,\tau+\epsilon)$.
This is a contradiction to the uniqueness of $\tau$. Infer from (\ref{unifconv}) that there exists a natural number $n_0(\epsilon)$ such that
\begin{equation} \label{sd}
||f_n-f|| \le \delta_\epsilon \quad \forall \; n \ge n_0(\epsilon).
\end{equation}
Let $t \notin U_\epsilon$, so $t \in V_\epsilon$. Notice that
\begin{equation} \label{incr}
 f_n(\tau)-f_n(t) = [f_n(\tau)-f(\tau)]+[f(\tau)-f(t)]+[f(t)-f_n(t)].
\end{equation}
By (\ref{sd}) the first summand $f_n(\tau)-f(\tau)$ and the third summand $f(t)-f_n(t)$ are greater or equal $-\delta_\epsilon$.
As to the second summand observe that $f(t) \le m_\epsilon$, because $t \in V_\epsilon$. Thus $f(\tau)-f(t) \ge f(\tau)-m_\epsilon= 3 \delta_\epsilon$. Summing up we arrive at $f_n(\tau)-f_n(t) \ge -2 \delta_\epsilon +3 \delta_\epsilon = \delta_\epsilon$ or equivalently
\begin{equation} \label{bineq}
 f_n(t) \le f_n(\tau)-\delta_\epsilon < f_n(\tau) \quad \forall \; t \notin U_\epsilon \quad \forall \; n \ge n_0(\epsilon).
\end{equation}
From this basic inequality we can derive that also
\begin{equation} \label{bineq2}
 f_n(t-) < f_n(\tau) \quad \forall \; t \notin U_\epsilon \quad \forall \; n \ge n_0(\epsilon).
\end{equation}
To see this consider at first the case $t \in (0,\tau-\epsilon]$. Then there exists a sequence $(s_k)$ with $0<s_k \uparrow t$, whence by (\ref{bineq}) applied to $t=s_k$ it follows that
$f_n(t-)=\lim_{k \rightarrow \infty} f_n(s_k) \le f_n(\tau)-\delta_\epsilon<f_n(\tau)$. In the same way one can treat the case $t \in (\tau+\epsilon,1]$ and finally if $t=0$, then $f_n(0-)=f_n(0)$ by definition and another application of (\ref{bineq}) gives (\ref{bineq2}).
Now, (\ref{bineq}) and (\ref{bineq2}) show that
\begin{equation} \label{bineq3}
\max\{f_n(t),f_n(t-)\} < f_n(\tau) \quad \forall \; t \notin U_\epsilon \quad \forall \; n \ge n_0(\epsilon).
\end{equation}
Conclude that
\begin{equation} \label{inU}
 \tau_n := a(f_n) \in U_\epsilon \quad \forall \; n \ge n_0(\epsilon),
\end{equation}
because otherwise there exists an $n \ge n_0(\epsilon)$ such that $\tau_n \notin U_\epsilon$. But since $\tau_n$ is the (smallest) supremizing point
of $f_n$ we obtain with (\ref{bineq3}) that
$M(f_n)= \max\{f_n(\tau_n),f_n(\tau_n-)\} < f_n(\tau)$, a contradiction to $M(f_n)$ is the (least) upper bound of $f_n$.
In the extreme cases $\tau=0$ or $\tau=1$ one considers $U_\epsilon$ equal to $[0,\epsilon]$ or $(1-\epsilon,1]$, respectively, and the same
modus operandi as above leads to (\ref{inU}). Thereby we have shown that $a(f_n) \rightarrow a(f)$ whenever $f_n \rightarrow_s f$, which
means that $a$ is continuous at $f$.
\end{proof}

\end{appendix}

\end{document}